\documentclass[12pt]{article}
\usepackage{amssymb,amsmath,amsthm,cite,enumerate}

\title{The nonexistence of an additive quaternary
$[15,5,9]$-code}
\author{J\"urgen Bierbrauer
\thanks{research partially supported by NSA grant H98230-10-1-0159}\\
Department of Mathematical Sciences\\
Michigan Technological University\\
Houghton, Michigan 49931 (USA)\\ \\
D. Bartoli, G. Faina, S. Marcugini and F. Pambianco
\thanks{partially supported by the Italian Ministero dell'Istruzione, dell'Universit\`a e della Ricerca (MIUR) and by the Gruppo Nazionale per le Strutture Algebriche, Geometriche e le loro Applicazioni (GNSAGA).} \\
Dipartimento di Matematica e Informatica\\
Universit\`a degli Studi di Perugia\\
Perugia (Italy)}

\begin{document}

\maketitle

\theoremstyle{plain}
\newtheorem{Theorem}{Theorem}[section]
\newtheorem{Proposition}[Theorem]{Proposition}
\newtheorem{Lemma}[Theorem]{Lemma}
\newtheorem{Corollary}[Theorem]{Corollary}

{\theoremstyle{definition}
\newtheorem{Definition}[Theorem]{Definition}
\newtheorem{Example}[Theorem]{Example}}

\def\gz{\mathbb{Z}}
\def\rz{\mathbb{R}}
\def\ef{\mathbb{F}}
\def\CC{\mathbb{C}}
\def\o{\omega}
\def\p{\overline{\omega}}
\def\e{\epsilon}
\def\a{\alpha}
\def\b{\beta}
\def\g{\gamma}
\def\d{\delta}
\def\l{\lambda}
\def\s{\sigma}
\def\bsl{\backslash}
\def\la{\longrightarrow}
\def\arr{\rightarrow}
\def\ov{\overline}
\def\sm{\setminus}
\newcommand{\D}{\displaystyle}
\newcommand{\T}{\textstyle}
\newcommand{\PG}{\mathrm{PG}}

\begin{abstract}
We show that no additive $[15,5,9]_4$-code exists. As a consequence the largest dimension $k$
such that an additive quaternary $[15,k,9]_4$-code exists is $k=4.5.$
\end{abstract}

\section{Introduction}

Additive codes generalize the notion of linear codes, see also
Chapter 17 of~\cite{book}.
Here we concentrate on the quaternary case.

\begin{Definition}
\label{addcodebasicdef}
Let $k$ be such that $2k$ is a positive integer.
An additive quaternary $[n,k]_4$-code ${\cal C}$ (length $n,$
dimension $k$) is a $2k$-dimensional subspace of $\ef_2^{2n},$
where the coordinates come in pairs of two.
We view the codewords as $n$-tuples where the coordinate entries are elements of $\ef_2^2.$

A {\em generator matrix} $G$ of ${\cal C}$ is a binary $(2k,2n)$-matrix whose
rows form a basis of the binary vector space ${\cal C}.$
\end{Definition}

Our main result is the following.

\begin{Theorem}
\label{nonexisttheorem}
There is no additive $[15,5,9]_4$-code.
\end{Theorem}

The theory of additive codes is a natural and far-reaching generalization of the classical
theory of linear codes. The classical theory of cyclic and constacyclic linear codes
has been generalized to additive codes in~\cite{addcyclic12}. In the sequel we restrict to the
case of quaternary additive codes and write $[n,k,d]$ for $[n,k,d]_4.$
The quaternary case is of special interest, among others because of a close link to the
theory of quantum stabilizer codes and their geometric representations, see~\cite{quantgeom,DHY,BMP,BBEFMP,BBMP}.
The determination of the optimal parameters of additive quaternary codes of short length
was initiated by Blokhuis and Brouwer~\cite{BB}. In~\cite{IEEEaddZ4} we determine the optimal
parameters for all lengths $n\leq 13$ except in one case. The last gap was closed by
Danielsen-Parker~\cite{DP} who constructed two cyclic $[13,6.5,6]$-codes.
Let us concentrate on lengths $n=14$ and $n=15$ now.
A cyclic $[15,4.5,9]$-code was constructed in~\cite{JBMadrid}, see also~\cite{addcyclic12}.
Together with Theorem~\ref{nonexisttheorem} some more optimal parameters are determined.
The following table collects this information.

$$\left.\begin{array}{c||c|c|c|c|c|c|c|c|c|c|c|c|c|c|c}
k\bsl n & 1 & 2 & 3 & 4 & 5 & 6 & 7 & 8 & 9 & 10 & 11 & 12 & 13 & 14 & 15 \\
\hline
1       & 1 & 2 & 3 & 4 & 5 & 6 & 7 & 8 & 9 & 10 & 11 & 12 & 13 & 14 & 15 \\
1.5     &   & 1 & 2 & 3 & 4 & 5 & 6 & 6 & 7 & 8  & 9  & 10 & 11 & 12 & 12 \\
2       &   & 1 & 2 & 3 & 4 & 4 & 5 & 6 & 7 & 8  & 8  & 9  & 10 & 11 & 12 \\
2.5     &   &   & 1 & 2 & 3 & 4 & 5 & 6 & 6 & 7  & 8  & 8  & 9  & 10 & 11  \\
3       &   &   & 1 & 2 & 3 & 4 & 4 & 5 & 6 & 6  & 7  & 8  & 9  & 10 & 11  \\
3.5     &   &   &   & 1 & 2 & 3 & 4 & 4 & 5 & 6  & 7  & 8  & 8  & 9  & 10 \\
4       &   &   &   & 1 & 2 & 2 & 3 & 4 & 5 & 6  & 6  & 7  & 8  & 9  & 10 \\
4.5     &   &   &   &   & 1 & 2 & 3 & 3 & 4 & 5  & 6  & 6  & 7  & 8  & 9  \\
5       &   &   &   &   & 1 & 2 & 2 & 3 & 4 & 5  & 6  & 6  & 7  & 8  & 8  \\
5.5     &   &   &   &   &   & 1 & 2 & 3 & 3 & 4  & 5  & 6  & 6  & 7  & 8  \\
6       &   &   &   &   &   & 1 & 2 & 2 & 3 & 4  & 5  & 6  & 6  & 7  & 8  \\
6.5     &   &   &   &   &   &   & 1 & 2 & 3 & 3  & 4  & 5  & 6  & 6  & 7  \\
7       &   &   &   &   &   &   & 1 & 2 & 2 & 3  & 4  & 4  & 5  & 6  & 7  \\
7.5     &   &   &   &   &   &   &   & 1 & 2 & 2  & 3  & 4  & 5  & 5-6& 6  \\
8       &   &   &   &   &   &   &   & 1 & 2 & 2  & 3  & 4  & 4  & 5  & 6  \\
8.5     &   &   &   &   &   &   &   &   & 1 & 2  & 2  & 3  & 4  &4-5 & 5-6 \\
9       &   &   &   &   &   &   &   &   & 1 & 2  & 2  & 3  & 4  & 4  & 5  \\
9.5     &   &   &   &   &   &   &   &   &   & 1  & 2  & 2  & 3  & 4  & 4-5 \\
10      &   &   &   &   &   &   &   &   &   & 1  & 2  & 2  & 3  & 4  & 4   \\
10.5,11 &   &   &   &   &   &   &   &   &   &    & 1  & 2  & 2  & 3  & 4  \\
11.5,12 &   &   &   &   &   &   &   &   &   &    &    & 1  & 2  & 2  & 3  \\
 \end{array}\right.$$

In order to understand the entries in the two last columns corresponding to
$n=14$ and $n=15,$ observe that an $[n,k,d]$-code implies an
$[n-1,k,d-1]$-code by puncturing and an $[n-1,k-1,d]$-code by shortening.
Together with the knowledge of the optimal parameters of linear quaternary
codes of short length (see for example~\cite{Grassl}), this determines most of
the entries. As an example consider the case of dimension $k=4.$ It is known
that $d=8$ is the optimal distance in case $n=13,k=4.$ This shows that
$[14,4,d_1]$ and $[15,4,d_2]$ have $d_1\leq 9$ and $d_2\leq 10.$ As a linear
quaternary $[15,4,10]$-code exists (it is derivable from the $[17,4,12]$-code
determined by the elliptic quadric) it follows that $d_1=9, d_2=10$ are the
optimal minimum distances in those cases.

\begin{Corollary}
\label{nonexistcor}
The maximum dimension of an additive quaternary $[15,k,9]$-code is
$k=4.5.$
\end{Corollary}

This follows from Theorem~\ref{nonexisttheorem} and the existence of a
$[15,4.5,9]$-code. This cyclic code is a direct sum $C=C_0\oplus C_1\oplus
C_2$ of codes of binary dimensions $1,4,4$. Here, $C_0$ is
generated by $(11)^{15}.$ Let $\ef_{16}=\ef_2(\e )$ where $\e^4=\e +1,$ let
$T:\ef_{16}\la\ef_2$ be the trace and index the coordinates of $C$ by $\e^i,
i=0,\dots ,14.$ The $16$ codewords of $C_1$ are indexed by $u\in\ef_{16},$
with entry $(T(u\e^{i+1}),T(u\e^i))$ in coordinate $\e^i.$ Likewise, the
codewords of $C_2$ are indexed by $u\in\ef_{16},$ with entry
$(T(u\e^{3i}),T(u\e^{3i+2}))$ in coordinate $\e^i.$ This code has strength
$3.$ Its automorphism group is the cyclic group of order $15.$

In the proof of Theorem~\ref{nonexisttheorem}
we are going to use the information contained in the table
with the exception of the non-existence of $[15,5,9]$ of course.
In geometric language Theorem~\ref{nonexisttheorem} is equivalent to the
following:

\begin{Theorem}
\label{geomnonexisttheorem}
There is no multiset of $15$ lines in $\PG(9,2)$ such that
no more than $6$ of those lines are in a hyperplane.
\end{Theorem}

We are going to prove Theorem~\ref{geomnonexisttheorem} in the following
sections, using geometric arguments and an exhaustive computer search.
Most of the geometric work is done in $\PG(9,2),$ a suitably chosen
subspace $\PG(5,2)$ and its factor space $\PG(3,2).$
Total computing time was 250 days using a 3.2 Ghz Intel Exacore.

\section{Some basic facts}
\label{basicgeomsection}

Let ${\cal M}$ be a multiset of $15$ lines in $\PG(9,2)$ with the property that no
hyperplane of $\PG(9,2)$ contains more than $6$ of those lines, in the multiset sense.
We will write $V_n$ for the subspaces $\PG(n-1,2)$ in our $\PG(9,2).$

\begin{Definition}
\label{strengthdef}
A set of $t$ lines is in {\em general position} if the space they generate has projective
dimension $2t-1.$
A set of lines has {\em strength} $t$ if $t$ is maximal such that any $t$ of the lines
are in general position.
\end{Definition}

The coding-theoretic meaning of the strength is the following.

\begin{Proposition}
\label{strengthdualprop}
If $C$ is an $[n,k]_4$ additive quaternary code geometrically described
by a set $L$ of $n$ lines in $\PG(2k-1,2),$ then the minimum distance of the dual
code $C^{\perp}$ equals $t+1,$ where $t$ is the strength of $L.$
\end{Proposition}

\begin{Lemma}
\label{firstbasiclemma}
The multiset ${\cal M}$ is a set and has strength $3.$
\end{Lemma}

\begin{proof}
Suppose ${\cal M}$ is not a set. This means that some two of its members
are identical; say $L_1=L_2.$ Consider the subcode consisting of the codewords
with entries $0$ in the positions of $L_1$ and $L_2.$ This yields a
$[13,4,9]_4$-code which is known not to exist.
In fact, concatenation of a $[13,4,9]_4$-code with the $[3,2,2]_2$-code
would produce a binary linear $[39,8,18]_2$-code, which contradicts the Griesmer bound.

Now, suppose that some three lines $L_1,L_2,L_3$ of ${\cal M}$ are not in general position.
The subcode with entries $0$ in those three coordinate positions has length
$12,$ dimension $\geq 5-2.5=2.5$ and minimum distance $\geq 9.$
Such a code does not exist. In fact, concatenation yields a binary linear
$[36,5,18]_2$-code, which again contradicts the Griesmer bound.
It follows that the strength is $\geq 3.$
Assume it is $\geq 4.$ Then the dual of the code $C$ described by ${\cal M}$
is a $[15,10,5]_4$-code by Proposition~\ref{strengthdualprop}.
Such a code does not exist. In fact, shortening results in a $[12,7,5]_4$-code
whose nonexistence has been shown in~\cite{IEEEaddZ4}.
It follows that the strength of ${\cal M}$ is precisely $3.$
\end{proof}

\begin{Definition}
\label{codepointdef}
\begin{enumerate}[\rm(1)]
\item The lines of ${\cal M}$ are called {\em codelines}.
\item A point in $\PG(9,2)$ is a {\em codepoint} if it is on some codeline.
\item If $M$ is the set of $45$ codepoints and
 $U$ be a subspace, then  the {\em weight} of $U$ is
 $w(U)=\vert U\cap M\vert .$
\item A subspace $V_i$ of weight $m$ is also called
an $m-V_i.$
\end{enumerate}
\end{Definition}

\begin{Lemma}
\label{secondbasiclemma}
\begin{enumerate}[\rm(1)]
\item There are at most
$27$ codepoints on a hyperplane$,\ 18$ points on a $V_8,$
$13$ points on a $V_7,$ $10$ points on a $V_6$ and $8$ points
on a $V_5.$
\item
All these upper bounds can be attained.

\item Each $18-V_8$ is contained in three $27$-hyperplanes.

\item There is a $10-V_6$ containing three
codelines and an isolated point.
\end{enumerate}

\end{Lemma}

\begin{proof}
Let $H$ be a hyperplane. Each codeline meets $H$ either
in $3$ points or in one point. Because of the defining
condition, $w(H)\leq 6\times 3+9\times 1=27.$
This bound is reached with equality, as otherwise ${\cal M}$
would describe a $[15,5,10]_4$-code, which does not exist.

Let $S$ be a $V_8$ and $w(S)=i.$
The distribution of the codepoints on the hyperplanes containing
$S$ shows that $i+3(27-i)\geq 45,$ which implies $i\leq 18.$ The argument
also implies that in case $i=18$ all hyperplanes containing
$S$ have weight $27.$

The remaining bounds follow from similar elementary counting arguments.
The $V_7$-bound follows from $14+7\times 4=42<45,$ the $V_6$-bound
from $11+15\times 2<45$ and the $V_5$-bound from $7+31<45.$

Let us show that all those bounds are reached.
As the strength is $3$ we find some four
lines in a $V_7.$ We know that they are not contained in a $V_6;$
so they generate the $V_7.$
Without loss of generality, these lines can be chosen as
$$\langle v_1,v_2\rangle ,~\langle v_3,v_4\rangle ,~\langle
v_5,v_6\rangle ,~\langle v_7,v_1+v_3+v_5\rangle .$$
This shows that
$\langle v_1,\dots ,v_5\rangle$ is an $8-V_5$ and
$\langle v_1,\dots ,v_6\rangle$ is a $10-V_6$ containing three
codelines and an isolated point. This in turn is contained in a
$13-V_7$ and so on.
\end{proof}

A more accurate count can be made using geometric arguments in factor spaces.

\begin{Definition}
\label{factorweightdef}
Let $U$ be a subspace and $\Pi_U$ the factor space.
If $U$ is a $V_i,$ then $\Pi_U$ is a $PG(9-i,2).$
Let $P\in\Pi_U$ be a point and $\langle U,P\rangle$
the $V_{i+1}$ whose factor space mod $U$ is the point $P.$
Define $w(P)=w(\langle U,P\rangle )-w(U).$
For each subspace $X\subseteq\Pi_U$ define $w(X)=\sum_{P\in X}w(P).$
\end{Definition}

Observe that the weight $w(P)$ equals the number of codepoints
which are in the preimage of $P$ mod $U$ but not in $U.$
In particular, $$\sum_{P\in\Pi_U}w(P)=45-w(U).$$

\begin{Lemma}
\label{13V7lemma}
Let $U_7$ be a $13-V_7$ and $\Pi_7$ the factor space (a Fano plane).
Then $\Pi_7$ has three collinear points of weight $4,$ the remaining
four points having weight $5.$
\end{Lemma}

\begin{proof}
By Lemma~\ref{secondbasiclemma},
each point of $\Pi_7$ has weight $\leq 5,$ each line
has weight $\leq 14$ and the sum of all weights is $45-13=32.$
Let $B$ be the set of points of weight $<5$ in $\Pi_7.$
Then $B$ is a blocking set
for the lines and $\vert B\vert\leq 3.$ It follows that $\vert B\vert =3$ and
$B$ is a line. This gives the result.
\end{proof}

\section{The geometric setting}

In Lemma~\ref{secondbasiclemma} we saw that there is a
$10-V_6$ containing $3$ codelines and an isolated point.
We concentrate on such a subspace and its factor space $PG(3,2).$

\begin{Definition}
\label{basicweightdef}
Let $U$ be a $10-V_6$ containing three codelines $L_1,L_2,L_3$ and
an isolated codepoint.
Let $P_0\in\Pi_U$ be the unique point whose preimage contains another
codeline $L_4.$
\par
For each line $l$ of $\Pi_U$ define the {\em $h$-weight} $h(l)$ as
the number of codelines
different from $L_1,\ldots, L_4$ which are in the preimage of $l$ mod $U.$
For each plane $E$, let $h(E)$ be the number of codelines different
from $L_1,\ldots, L_4$ which are contained in the preimage of $E.$
\end{Definition}

\begin{Lemma}
\label{10V6lemma}
\begin{enumerate}[\rm(1)]
\item The points of $\Pi_U$ have weights $\leq 3,$ lines have weights $\leq 8,$
planes have odd weights $\leq 17.$

\item Each $8$-line of $\Pi_U$ is contained in three $17$-planes.

\item Each $17$-plane of $\Pi_U$ contains either precisely $3$ or precisely $4$
points of weight $3.$ There is no point of weight $0$ on a $17$-plane.
\end{enumerate}

\end{Lemma}
\begin{proof} (1) This follows from Lemma~\ref{secondbasiclemma}.

(2) Let $E$ be a plane of $\Pi_U$ and $H$ the preimage of $E.$
Then $H$ is a hyperplane of $PG(9,2).$ If $H$ contains $i\leq 6$ codelines,
then $w(H)=15+2i$ and $w(E)=w(H)-10=5+2i,$ an odd number $\leq 17.$
Lemma~\ref{secondbasiclemma} shows that each $8$-line
of $\Pi_U$ is contained in three $17$-planes.

(3) Let $E$ be a $17$-plane. As $7\times 2=14,$ the plane $E$ contains at least
three points of weight $3.$ Assume it has five such points. Then there is a
line $g$ all of whose points have weight $3.$ This is a line of weight
$9,$ contradiction.
\end{proof}

\begin{Lemma}
\label{10V6secondlemma}
All points of $\Pi_U$ have weights $1,2$ or $3.$
\end{Lemma}

\begin{proof}
Assume that there is a point $P$ of weight $0.$ Let $P\in l.$ As every plane
containing $l$ has weight at most $15$ we must have $w(l)\leq 5.$ The pencil
of lines through $P$ shows that all lines $l$ through $P$ must have weight
$5.$ It follows that there are $7$ points of weight $3$ and $7$ of weight $2.$
Consider the line $h$ through two of the points of weight $3.$ The third point
on that line must have weight $2$ or $3.$ The latter is excluded as $w(h)\leq
8.$ It follows that $w(h)=8.$ The plane generated by $P$ and the line $h$ of
weight $8$ has weight $15$, which contradicts a statement from
Lemma~\ref{secondbasiclemma}.
\end{proof}

\begin{Definition}
\label{xidef}
Let $m_i$ be the number of points of $\Pi_U$ of weight $i.$
\end{Definition}

In particular, $m_1+m_2+m_3=15.$

\begin{Lemma}
\label{w3lemma}
A point $P$ of weight $3$ is contained in one plane of weight $15$ and in
six planes of weight $17.$
It is contained in three lines of weight $7$ and in four lines of weight $8.$
The plane of weight $15$ is the union of $P$ and the lines of weight $7$ through $P.$
\end{Lemma}

\begin{proof}
Consider the pencil of $7$ lines through $P.$ If $P$ is on a plane of weight
$13,$ then the sum of the weights of the points on some three of the lines is
$13.$ As the total weight is $35,$ one of the remaining four lines of the
pencil must have weight $>8,$ contradiction. It follows that the planes
containing $P$ have weights $15$ or $17.$ Assume they all have weight $17.$
Then all lines through $P$ must have weight $8,$ which leads to the
contradiction $35=3+7\times 5.$ It follows that the weights of the points on
some three lines of the pencil add to $15.$ The remaining four lines of the
pencil have weight $8$ each. It follows that all remaining planes containing
$P$ have weight $17.$ Also, any two lines of weight $8$ through $P$ determine
a line of weight $7$ as the third line through $P$ on the same plane (of
weight $17$).
\end{proof}

\begin{Lemma}
\label{heavylemma}
\begin{enumerate}[\rm(1)]
\item
The sum of all $h$-weights of lines is $11.$
\item
The sum of all $h$-weights of lines through $P_0$ is $w(P_0)-2.$
\item
For $P\neq P_0$ the sum of the $h$-weights of lines through $P$
is $w(P).$
\item
Let $E$ be a plane not containing $P_0.$ Then the sum of the $h$-weights of lines in $E$ is $(w(E)-11)/2.$
\item
Let $E$ be a plane containing $P_0.$ Then the sum of the $h$-weights of lines in $E$ is $(w(E)-13)/2.$
\end{enumerate}
\end{Lemma}
\begin{proof} (1) This is immediate.

(3) This follows from
the fact that each of the remaining $11$ lines different from $L_1,\dots ,L_4$
whose image mod $U$ passes through $P\not=P_0$ contributes exactly one
codepoint to the weight of $P.$ When $P=P_0$ two of the codepoints
contributing to $w(P_0)$ come from line $L_4.$ This shows (2).

(4) The number of codelines contained in the preimage $H$ of $E$ is then $3+h(E).$ It follows
$$
|H \cap M|= 3(3+h(E))+(15-3-h(E))=2h(E)+21=10+w(E),
$$
which implies $h(E)=(w(E)-11)/2$, as claimed.

(5) In the case
when $P_0\in E$ the argument is analogous.
\end{proof}

In particular each plane has odd weight between $11$ and $17.$
In what follows, three cases are distinguished.

\begin{enumerate}
\item[] {\bf Case 1:} There is a plane of weight $11.$

\item[] {\bf Case 2:} There is a plane of weight $13.$

\item[] {\bf Case 3:} All planes have weight $15$ or $17.$
\end{enumerate}

In each case we attempt to construct a generator matrix $G$ of the code $C.$
Here $G$ has $10$ rows and $15$ pairs of columns. The column pairs correspond
to codelines, where each codeline is represented by two of its three points.
The top $6$ rows correspond to the parameters
$x_1,\dots ,x_6$ of $U,$ the remaining four rows to the parameters $y_1,y_2,y_3,y_4$
of $\Pi_U,$ where all those parameters are read from top to bottom.
In each case we start from our lines $L_1,L_2,L_3,L_4$ and a system of $11$
lines in $\Pi_U=PG(3,2)$ which determine the last $4$ rows of $G.$
A computer program then decides that this cannot be completed to a generator
matrix of a code with minimum distance $9.$

\section{Case 1}
Assume there is a plane $E_0$ of weight $11.$ Then all points in $\Pi_U\sm
E_0$ have weight $3.$ All points of $E_0$ have weight $1$ or $2$ as otherwise
a line of weight $9$ would exist. It follows that $E_0$ has $3$ points of
weight $1$ and $4$ points of weight $2,$ consequently $m_1=3,m_2=4,m_3=8.$ No
three points of weight $2$ can be collinear as otherwise some plane would have
even weight. It follows that the four points of weight $2$ form a quadrangle
in $E_0,$ and those of weight $1$ are collinear on a line $l_0.$ The planes
$E\not=E_0$ intersecting $E_0$ in $l_0$ have weight $15,$ the others have
weight $17.$ Write $E_0$ as $y_1=0.$ The points of weight $2$ are
$0100,0010,0001,0111,$ the remaining points of $E_0$ having weight $1$ form
the line $l_0=\lbrace 0110,0101,0011\rbrace$ and all affine points have weight
$3.$ As $w(E_0)=11,$ the preimage of $E_0$ contains $3$ codelines. It follows
that all lines of $E_0$ have $h$-weight $0$ and the point $P_0$ which defines
a $V_7$ with four lines is not on $E_0.$ We choose $P_0=1000$ as the
corresponding point, and $L_4=\langle v_1+v_3+v_5,0^61000\rangle .$ As
$w(P_0)=3$ there is a unique line $g_0$ of $h$-weight $1$ through $P_0.$ The
line $g_0$ through $P_0$ of $h$-weight $1$ can be chosen as either $P_0\cdot
0100$ or $P_0\cdot 0110.$

A computer calculation shows that there are $12$ solutions for the resulting
system of $11$ lines in $\Pi_U.$ In all cases $g_0=P_0\cdot 0110.$

\begin{center}
\begin{figure}
\label{Figure1}
\begin{center} \caption{Case 1, the first $6$ solutions} \end{center}

$$\begin{tabular}{|c|c|c|c|c|c|c|c|c|c|c|} \hline
$L_5$&$L_6$&$R_1$&$R_2$&$R_3$&$R_4$&$R_5$&$R_6$&$R_7$&$R_8$&$R_9$ \\ \hline
10 & 01 & 01 & 01 & 01 & 01 & 01 & 01 & 01 & 01 & 01 \\
01 & 01 & 01 & 11 & 01 & 00 & 11 & 11 & 11 & 11 & 00 \\
01 & 01 & 10 & 10 & 10 & 01 & 10 & 00 & 01 & 00 & 11 \\
00 & 10 & 01 & 10 & 00 & 10 & 11 & 10 & 01 & 01 & 10 \\ \hline
10 & 01 & 01 & 01 & 01 & 10 & 01 & 11 & 11 & 01 & 01 \\
01 & 01 & 01 & 01 & 00 & 11 & 00 & 10 & 10 & 11 & 11 \\
01 & 00 & 01 & 10 & 11 & 00 & 11 & 01 & 01 & 01 & 01 \\
00 & 10 & 10 & 00 & 01 & 01 & 10 & 10 & 10 & 01 & 01 \\ \hline
10 & 01 & 01 & 01 & 01 & 11 & 10 & 01 & 10 & 01 & 01 \\
01 & 01 & 01 & 01 & 00 & 10 & 11 & 00 & 11 & 11 & 11 \\
01 & 00 & 10 & 10 & 01 & 01 & 00 & 11 & 11 & 01 & 01 \\
00 & 10 & 01 & 00 & 10 & 10 & 01 & 10 & 01 & 01 & 01 \\ \hline
10 & 01 & 01 & 01 & 10 & 01 & 01 & 01 & 01 & 10 & 10 \\
01 & 01 & 10 & 00 & 11 & 11 & 00 & 01 & 01 & 11 & 11 \\
01 & 01 & 01 & 01 & 00 & 00 & 11 & 10 & 10 & 01 & 01 \\
00 & 10 & 00 & 10 & 01 & 01 & 10 & 01 & 01 & 01 & 01 \\ \hline
10 & 01 & 01 & 01 & 01 & 10 & 01 & 11 & 10 & 01 & 01 \\
01 & 01 & 01 & 01 & 10 & 11 & 00 & 10 & 11 & 11 & 00 \\
01 & 00 & 01 & 10 & 01 & 01 & 11 & 01 & 00 & 01 & 11 \\
00 & 10 & 10 & 01 & 00 & 01 & 01 & 10 & 01 & 01 & 10 \\ \hline
10 & 01 & 01 & 10 & 01 & 10 & 01 & 01 & 10 & 01 & 01 \\
01 & 01 & 10 & 11 & 00 & 11 & 11 & 00 & 11 & 01 & 01 \\
01 & 00 & 01 & 01 & 01 & 00 & 01 & 11 & 11 & 10 & 10 \\
00 & 10 & 00 & 01 & 10 & 01 & 01 & 10 & 01 & 01 & 01 \\ \hline
\end{tabular}$$

\end{figure}
\end{center}


\begin{center}
\begin{figure}
\label{Figure2}
\begin{center} \caption{Case 1, the remaining $6$ solutions} \end{center}
$$\begin{tabular}{|c|c|c|c|c|c|c|c|c|c|c|} \hline
$L_5$&$L_6$&$R_1$&$R_2$&$R_3$&$R_4$&$R_5$&$R_6$&$R_7$&$R_8$&$R_9$ \\ \hline
10 & 01 & 01 & 01 & 01 & 10 & 01 & 10 & 01 & 11 & 11 \\
01 & 01 & 01 & 00 & 00 & 11 & 11 & 10 & 11 & 10 & 10 \\
01 & 01 & 10 & 01 & 11 & 00 & 01 & 01 & 00 & 01 & 01 \\
00 & 10 & 00 & 10 & 01 & 01 & 01 & 01 & 01 & 10 & 10 \\ \hline
10 & 01 & 01 & 11 & 10 & 01 & 10 & 01 & 10 & 01 & 01 \\
01 & 01 & 01 & 10 & 11 & 11 & 10 & 11 & 11 & 00 & 00 \\
01 & 10 & 10 & 01 & 00 & 01 & 01 & 00 & 11 & 01 & 01 \\
00 & 01 & 00 & 10 & 01 & 01 & 01 & 01 & 01 & 10 & 10 \\ \hline
10 & 01 & 01 & 01 & 10 & 01 & 01 & 11 & 10 & 10 & 01 \\
01 & 01 & 01 & 10 & 11 & 00 & 00 & 10 & 11 & 10 & 11 \\
01 & 01 & 10 & 01 & 01 & 01 & 11 & 01 & 00 & 01 & 00 \\
00 & 10 & 01 & 00 & 01 & 10 & 01 & 10 & 01 & 01 & 01 \\ \hline
10 & 01 & 10 & 10 & 10 & 01 & 10 & 01 & 01 & 01 & 01 \\
01 & 10 & 11 & 11 & 10 & 11 & 11 & 01 & 01 & 00 & 00 \\
01 & 01 & 01 & 00 & 01 & 00 & 11 & 10 & 10 & 01 & 01 \\
00 & 00 & 01 & 01 & 01 & 01 & 01 & 01 & 01 & 10 & 10 \\ \hline
10 & 01 & 01 & 01 & 10 & 01 & 10 & 01 & 01 & 11 & 11 \\
01 & 01 & 01 & 10 & 11 & 11 & 10 & 00 & 00 & 10 & 10 \\
01 & 00 & 01 & 01 & 00 & 01 & 01 & 11 & 11 & 01 & 01 \\
00 & 10 & 10 & 00 & 01 & 01 & 01 & 01 & 01 & 10 & 10 \\ \hline
10 & 01 & 01 & 01 & 01 & 01 & 11 & 10 & 01 & 10 & 10 \\
01 & 01 & 01 & 10 & 00 & 00 & 10 & 11 & 11 & 10 & 11 \\
01 & 00 & 10 & 01 & 01 & 11 & 01 & 00 & 01 & 01 & 11 \\
00 & 10 & 01 & 00 & 10 & 01 & 10 & 01 & 01 & 01 & 01 \\ \hline
\end{tabular}$$

\end{figure}
\end{center}

We illustrate with the first of these $12$ solutions. The generator matrix $G$
looks as follows.

$$\left[\begin{array}{c|c|c|c|c|c|c|c|c|c|c|c|c|c|c}
L_1&L_2&L_3&L_4&L_5&L_6&R_1&R_2&R_3&R_4&R_5&R_6&R_7&R_8&R_9 \\
10 & 00 & 00 & 10 &    &    &    &    &    &    &    &    &    &    &  \\
01 & 00 & 00 & 00 &    &    &    &    &    &    &    &    &    &    &  \\
00 & 10 & 00 & 10 &    &    &    &    &    &    &    &    &    &    &  \\
00 & 01 & 00 & 00 &    &    &    &    &    &    &    &    &    &    &   \\
00 & 00 & 10 & 10 &    &    &    &    &    &    &    &    &    &    &   \\
00 & 00 & 01 & 00 &    &    &    &    &    &    &    &    &    &    &   \\ \hline
00 & 00 & 00 & 01 & 10 & 01 & 01 & 01 & 01 & 01 & 01 & 01 & 01 & 01 & 01 \\
00 & 00 & 00 & 00 & 01 & 01 & 01 & 11 & 01 & 00 & 11 & 11 & 11 & 11 & 00 \\
00 & 00 & 00 & 00 & 01 & 01 & 10 & 10 & 10 & 01 & 10 & 00 & 01 & 00 & 11 \\
00 & 00 & 00 & 00 & 00 & 10 & 01 & 10 & 00 & 10 & 11 & 10 & 01 & 01 & 10   \\
\end{array}\right]$$

Use Gaussian elimination: as the bottom part of lines $L_5,R_1$ has full rank $4$
we can manage to have the upper part of those two lines identically zero.
This will not change the left column of $L_4,$ but the right column is destroyed.
The following form is obtained. Here we can choose the top entry of the right
column of $L_4$ to be zero as we could replace it by the sum of the columns.

$$\left[\begin{array}{c|c|c|l|c|c|c|c|c|c|c|c|c|c|c}
L_1&L_2&L_3&L_4&L_5&L_6&R_1&R_2&R_3&R_4&R_5&R_6&R_7&R_8&R_9 \\
10 & 00 & 00 & 10 & 00 &    & 00 &    &    &    &    &    &    &    &  \\
01 & 00 & 00 & 0  & 00 &    & 00 &    &    &    &    &    &    &    &  \\
00 & 10 & 00 & 1  & 00 &    & 00 &    &    &    &    &    &    &    &  \\
00 & 01 & 00 & 0  & 00 &    & 00 &    &    &    &    &    &    &    &   \\
00 & 00 & 10 & 1  & 00 &    & 00 &    &    &    &    &    &    &    &   \\
00 & 00 & 01 & 0  & 00 &    & 00 &    &    &    &    &    &    &    &   \\ \hline
00 & 00 & 00 & 01 & 10 & 01 & 01 & 01 & 01 & 01 & 01 & 01 & 01 & 01 & 01 \\
00 & 00 & 00 & 00 & 01 & 01 & 01 & 11 & 01 & 00 & 11 & 11 & 11 & 11 & 00 \\
00 & 00 & 00 & 00 & 01 & 01 & 10 & 10 & 10 & 01 & 10 & 00 & 01 & 00 & 11 \\
00 & 00 & 00 & 00 & 00 & 10 & 01 & 10 & 00 & 10 & 11 & 10 & 01 & 01 & 10   \\
\end{array}\right]$$

A computer search showed that this cannot be completed to a generator matrix
of the putative $[15,5,9]$-code. The same procedure excludes the remaining
$11$ solutions.

\section{Case 2}
Assume there is a plane $E_0$ of weight $13.$ By Lemma~\ref{w3lemma}, $E_0$
does not contain points of weight $3.$ It follows that $E_0$ contains a unique
point of weight $1,$ all remaining points having weight $2.$ Let $l\subset E$
be a line of type $(2,2,2).$ Then $l$ is contained in two $17$-planes, each
containing one further point of weight $2.$ It follows $m_1=1,m_2=8,m_3=6.$
Let $l'\subset E$ be a line of type $(1,2,2).$ Then $l'$ is on exactly one
$15$-plane, and this plane contains the two affine points (off $E$) of weight
$2.$ They are collinear with the weight $1$ point. This determines the
distribution of weights uniquely. There are three planes of weight $15$ and
eleven of weight $17.$ The latter come in three containing a weight $1$ point
and eight of the other type ($3$ points of weight $3$ and four of weight $2$).
Here is a concrete description: the weight $1$ point is $0100,$ the remaining
points on the weight $13$ plane $E_0=(y_1=0)$ have weight $2.$ The two
remaining points of weight $2$ are $1000,1100$ and the remaining six points
are of weight $3.$ The point $P_0$ which defines a $V_7$ with four lines must
have weight $\geq 2.$ It follows that we have without loss of generality three
subcases:

\begin{enumerate}
\item[] {\bf Subcase (2,1):} $P_0=1000$ (of weight $2,\notin E_0$)
\item[] {\bf Subcase (2,2):} $P_0=0010$ (of weight $2,\in E_0$)
\item[] {\bf Subcase (2,3):} $P_0=1010$ (of weight $3$)
\end{enumerate}

Subcase $(2,1)$ has $12$ solutions, subcase $(2,2)$ has $40$ solutions and subcase $(2,3)$ has $101$ solutions. In each of those $153$ cases a computer program checked that there is no completion to a generator matrix of a $[15,5,9]_4$-code.

\section{Case 3}
We have $m_1+m_2+m_3=15,~m_1+2m_2+3m_3=35$ which implies
$m_3=5+m_1.$ As lines have weight $<9,$ the points of weight $3$ form a cap in $PG(3,2).$
Observe also that each cap of size $6$ or more is contained in an affine space (avoids some plane). This shows that in case $m_1>0$ there is a
subplane without points of weight $3.$ Such a plane has weight $\leq 13,$ contradiction. We have $m_1=0,$ hence $m_2=10,~m_3=5.$
The only $5$-cap in $PG(3,2),$ which is not affine, consists
of $5$ points any four of which are in general position, a
coordinate frame which we choose as $1000, 0100, 0010, 0001, 1111.$
Those are the points of weight $3,$ all others have weight $2.$
There are $10$ planes of weight $17$ (containing three points of the frame)
and $5$ planes of weight $15$ (containing one point of the frame).
The point $P_0$ which describes a $V_7$ with four lines may have weight $2$ or $3.$ Observe that the stabilizer of the frame in $GL(4,2)$ is the full $S_5.$
This leads to two subcases as follows:

\begin{enumerate}
\item[] {\bf Subcase (3,1):} $P_0=1100$ (of weight $2$)
\item[] {\bf Subcase (3,2):} $P_0=1000$ (of weight $3$)
\end{enumerate}

There are $43$ solutions in Subcase $(3,1)$ and $70$ solutions in
Subcase $(3,2).$
In each of those cases a computer program checked that there is no completion to a generator matrix of a $[15,5,9]_4$-code.

We observe that in all three cases consideration of the tenth line of the generator matrix has never been necessary.
In other words, it never happened that the nine rows that have been chosen formed the generator matrix of a $[15,4.5,9]$-code.

\end{document}